\documentclass[12pt]{amsart}

\usepackage[colorlinks=true,linkcolor=blue]{hyperref}
\usepackage{amsmath}
\usepackage{amssymb}
\usepackage{amsthm}
\usepackage{mathtools}

\usepackage{graphicx}
\usepackage{epstopdf}
\usepackage{epsfig}

\usepackage{enumitem}

\usepackage{hyperref}

\usepackage{times}
\usepackage{relsize}
\usepackage{textcomp}
\usepackage{amssymb}
\usepackage[english]{babel}
\usepackage[autostyle]{csquotes}
\usepackage{epstopdf}
\usepackage{nicefrac}
\usepackage{tikz}

\theoremstyle{plain} 
\newtheorem{thm}{Theorem}[section]
\newtheorem{prop}[thm]{Proposition}
\newtheorem{lemma}[thm]{Lemma}
 
\newtheorem{question}[thm]{Question}
\newtheorem{conj}[thm]{Conjecture}
\theoremstyle{remark}

\theoremstyle{definition}
\newtheorem{defin}[thm]{Definition}
\newtheorem{claim}[thm]{Claim}

\newcommand{\la}{\langle}
\newcommand{\ra}{\rangle}

\begin{document}
	
	\title{A local-global conjugacy question arising from arithmetic dynamics}
	
	
	\author[V.Goksel]{Vefa Goksel}
	\address{Mathematics \& Statistics Department\\ University of Massachusetts\\
		Amherst\\
		MA 01003, USA}
	\email{goksel@math.umass.edu}

	\subjclass[2000]{Primary }
	
	
	\date{}
	
	\dedicatory{}
	
	\begin{abstract}
    In his earlier work, the author introduced a group theory question that arises in the study of iterated Galois groups of post-critically finite quadratic polynomials. In this paper, we prove the first non-trivial results on this question.
	\end{abstract}
	
	\maketitle
	\section{Introduction}
	Let $K$ be a field, and $f\in K[x]$ a quadratic polynomial. We define the $n$th iterate of $f$ inductively by setting $f^n(x) = f(f^{n-1}(x))$ for $n\geq 1$, where we make the convention that $f^0(x)=x$. Suppose that all iterates of $f$ are separable over $K$. Then $f^n$-pre-images of $0$ form a complete binary rooted tree as follows: For each root of $\alpha$ of $f^{n-1}$, we draw edges from $\alpha$ to two roots $\beta_1,\beta_2$ of $f^n$, where $f(\beta_1) = f(\beta_2) = \alpha$. We will denote this \emph{pre-image tree} of $f$ by $T$. We also let $T_n$ be the subtree of the first $n$ levels of $T$. The absolute Galois group naturally acts on $T$, which gives a continuous homomorphism
	
	\[\rho: \text{Gal}(K^{\text{sep}}/K) \rightarrow \text{Aut}(T).\]
	This action on the binary rooted tree is the dynamical analog of the Galois action on the $\ell$-power torsion points on an elliptic curve. For this reason, $\rho$ is called an \emph{arboreal Galois representation}, which is a terminology introduced by Boston and Jones in 2007 (\cite{BJ07}).
	
	If we let $G_n(f)$ be the Galois group of $f^n$ over $K$, these Galois groups give an inverse system via the natural surjections $G_{n+1}(f)\twoheadrightarrow G_n(f)$. Letting $G(f):=\varprojlim G_n(f)$, we have $\text{im}(\rho) = G(f)$. The question of describing this image as a subgroup of $\text{Aut}(T)$ is one major open question in arithmetic dynamics. There has been a lot of recent work on this topic, see for instance \cite{Ahmad,BJ07,BJ09,BenJuul19,Benedetto20,Bridy19,FPC19,Ferraguti19,Gok19,Gottesman10,Gratton,Hamblen,Hindes18,Hindes2018,Ingram13,Jones08,Jones14,Juul19,Odoni1,Odoni2,Odoni3,Pink,Stoll} for a limited list. Based on the existent results, it is believed that $G(f)$ should have finite index in $\text{Aut}(T)$ for most quadratic polynomials $f$. See \cite[Conjecture 3.11]{Jones14} for a precise conjecture in this direction. However, similar to the fact that elliptic curves with complex multiplication are excluded in Serre's celebrated open image theorem, there are some well-known exceptional families of $f$ for which $G(f)$ has infinite index in $\text{Aut}(T)$. One such exceptional case is when $f$ is \emph{post-critically finite} (see below for a definiton).\par
	
	A polynomial $f$ is called \emph{post-critically finite} (or PCF in short) if all its critical points have finite orbits under iteration by $f$. The exact description of the profinite group $G(f)$ as a subgroup of $\text{Aut}(T)$ when $f$ is PCF is still mysterious in general, and partial results are known only in some special cases. See for instance \cite{Ahmad,BJ09,BJ07,Benedetto_et_al17,Benedetto20} on this topic. In \cite{Gok19}, for certain PCF quadratic polynomials $f$, the author used a Markov model to construct special profinite subgroups $M(f)$ of $\text{Aut}(T)$ which conjecturally contain $G(f)$. If this containment can be established, the author believes that the groups $M(f)$ can be a good candidate for a dynamical analog of Mumford-Tate groups in the classical theory of Galois representations for elliptic curves. The author gave some evidence indicating that a positive answer to certain special cases of a purely group theoretic question will imply this containment. In this paper, we will prove the first non-trivial results on this question. \par
	
	Before we introduce the aforementioned group theory question, we first briefly give some notation and definitions: The groups $\text{Aut}(T_n)$ form an inverse system via the natural surjections $\pi_n: \text{Aut}(T_n)\twoheadrightarrow \text{Aut}(T_{n-1})$. We set $K_n=\text{Ker}(\pi_n)$. It is well-known that $K_n$ is an elementary abelian $2$-group with rank $2^{n-1}$, i.e. we have $K_n\cong (\mathbb{Z}/2\mathbb{Z})^{2^{n-1}}$.
    \begin{defin}
	\label{def1.1}
	Let $H,G\leq \text{Aut}(T_n)$ be two subgroups of $\text{Aut}(T_n)$. We say that \emph{$H$ is elementwise $K_n$-conjugate into $G$} if each element of $H$ can be conjugated into $G$ by an element of $K_n$. We say that \emph{$H$ is globally $K_n$-conjugate into $G$} if $H$ can be conjugated into $G$ by an element of $K_n$.
	\end{defin}
    \begin{defin}
    \label{def1.2}
    Let $H,G\leq \text{Aut}(T_n)$ be two subgroups of $\text{Aut}(T_n)$. We say that \emph{$\mathcal{P}(H,G)$ holds} if we have
    $$H \text{ is elementwise }K_n\text{-conjugate into }G \iff H \text{ is globally }K_n\text{-conjugate into }G.$$
    \end{defin}
    Definition \ref{def1.2} naturally raises the question that we will study:
    \begin{question}
    \label{Q1.3}
    Let $n\geq 1$. For which subgroups $H,G\leq \text{Aut}(T_n)$ does $\mathcal{P}(H,G)$ hold?
    \end{question}
Question \ref{Q1.3} seems difficult in general. It would also be perhaps interesting to note that similar questions in the context of Lie groups already exist in the literature, which also have applications to number theory, particularly to Langland's global functoriality conjecture. See for instance \cite{Chen,Larsen,Larsen2,Wang1,Wang2,Wang3}. Specifically, compare the question in \cite[p.99]{Chen} with Question~\ref{Q1.3}.\par

Note that $\mathcal{P}(H,G)$ trivially holds when $H$ is a cyclic subgroup of $\text{Aut}(T_n)$. We now state the first non-trivial result regarding Question \ref{Q1.3}.
\begin{thm}
	\label{thm1.4}
Let $n\geq 1$ and $H,G\leq \text{Aut}(T_n)$. Suppose that $|H|=|G|$ and $H\cap K_n = \{\text{id}\}$. Then $\mathcal{P}(H,G)$ holds.
\end{thm}
The proof of Theorem \ref{thm1.4} requires one to use the iterated wreath product structure of $\text{Aut}(T_n)$.\par
The structure of the paper will be as follows: In Section~\ref{sec:sec2}, we will give some background and motivation for Question~\ref{Q1.3}. In Section~\ref{sec:sec3}, we will present some preliminaries from group theory. In Section~\ref{sec:sec4}, we will prove some auxiliary lemmas that will be crucial in the proof of Theorem \ref{thm1.4}. We will prove Theorem \ref{thm1.4} in Section~\ref{sec:sec5}. We will finish the paper by stating a conjecture in Section~\ref{sec:sec6}.
\section{Background and motivation}
\label{sec:sec2}
Let $f\in \mathbb{Z}[x]$ be a monic PCF quadratic polynomial. It is well-known (and easy to prove) that all such polynomials are conjugates of $x^2$, $x^2-1$ or $x^2-2$ by the linear map $x\rightarrow x+a$ for $a\in \mathbb{Z}$. Let $G_n(f)$ denote the Galois group of the $n$-th iterate of $\mathbb{Q}(i)$, where we choose the base field as $\mathbb{Q}(i)$ for a technical reason, as explained in \cite[Section 2]{Gok19}. In \cite{Gok19}, the author used the factorization data of iterates of $f$ over finite fields together with a Markov process to construct the \emph{Markov groups} $M_n(f)$, and conjectured that $G_n(f)$ is contained in $M_n(f)$ for $n\geq 1$. Moreover, the author made a connection between Question \ref{Q1.3} and this conjecture. Namely, the author gave an argument indicating that $G_n(f)$ is elementwise $K_n$-conjugate into $M_n(f)$ for all $n$. If this is true, then establishing that $\mathcal{P}(G_n(f),M_n(f))$ holds for all $n$ will prove the conjecture above. Note that groups $M_n(f)$ are explicitly known, so this is a quite special case of Question~\ref{Q1.3}.\par

To give an example of groups $M_n(f)$, let $f:=(x+a)^2-a-1$ for some $a\in \mathbb{Z}$. This family is particularly interesting because it does not come from endomorphisms of algebraic groups. Recall that for $n\geq 1$, the permutations $a_1=(1,2)$, $a_2=(1,3)(2,4)$, \dots, $a_n=(1,2^{n-1}+1)(2,2^{n-1}+2)\cdots (2^{n-1},2^n)$ are the standard generators of $\text{Aut}(T_n)$. Then, $x_n:=a_1a_2\cdots a_n \in\text{Aut}(T_n)$ acts transitively on the $n$-th level of the tree \cite[Lemma 3.3]{Gok19}. One can also think of $x_i$ for $i<n$ as an element of $\text{Aut}(T_n)$, by using the natural inclusion $\text{Aut}(T_i)\hookrightarrow \text{Aut}(T_n)$. Finally, define $m_1:=\text{id}\in \text{Aut}(T_1)$, and $m_{n+1}:=x_n^2m_nx_n^{-1}$ for $n\geq 1$. Then, if $a\neq \pm b^2$ for any $b\in \mathbb{Z}$, we have $M_1(f)=\langle (1,2)\rangle$, $M_2(f) = \langle (1,3,2,4),(1,2)\rangle$, and $M_n(f)$ is given by 
\begin{equation}
\label{eq:Markov_group}
M_n(f)=\langle x_n,m_n,x_{n-1}^2,x_{n-2}^2,\dots,x_2^2\rangle
\end{equation} 
for $n\geq 3$. If $a = \pm b^2$ for some $b\in \mathbb{Z}$, then $M_n(f)$ is an explicit index-$2$ subgroup of the group given in (\ref{eq:Markov_group}) (See \cite[Corollary 6.28]{Gok19}). For $n\geq 1$, the author constructed the groups $M_n(f)$ inductively, by introducing the so-called \emph{Markov map}, which can be applied to above generators of $M_n(f)$ to obtain generators of $M_{n+1}(f)$. See \cite[Section 6]{Gok19} for more details.

It also follows from the work of the author in \cite{Gok19} that the parameters $a$ used above can be taken in $\mathbb{Q}$, and the main results of \cite{Gok19} will still hold. The author conjectures that the groups $M_n(f)$, the so called \emph{Markov groups}, exist for \emph{any} PCF quadratic polynomial over any number field, and that $G_n(f)$ is always contained in $M_n(f)$ for every $n$. This conjecture is based on extensive Magma computations combined with the theoretic evidence given in \cite{Gok19}, and is an ongoing project of the author. The author believes that there is potentially a rich theory behind these Markov groups, which may eventually lead to a dynamical analog of Mumford-Tate groups, as alluded to in the previous section.
\section{Preliminaries}
\label{sec:sec3}
In this section, we will recall some standard facts from group theory that will be necessary in the proof of Theorem \ref{thm1.4}. We set $W_n=\text{Aut}(T_n)$ for simplicity, which is the notation that we will be using throughout the paper. It is well-known that $W_n$ is isomorphic to the $n$-fold wreath product of $C_2$ (the cyclic group of order $2$), which can be inductively given by $W_1\cong C_2$ and $W_n\cong W_{n-1}\wr C_2$ for $n\geq 2$. By definition, this allows us to identify $W_n$ with the semi-direct product $\mathbb{F}_2^{2^{n-1}}\rtimes W_{n-1}$, where $\mathbb{F}_2^{2^{n-1}}$ stands for the $2^{n-1}$-dimensional vector space defined over $\mathbb{F}_2$, the finite field with $2$ elements. In this semi-direct product, $W_{n-1}$ acts on $\mathbb{F}_2^{2^{n-1}}$ simply by permuting the coordinates. Concretely, for $v=(v_1,\dots,v_{2^{n-1}})\in \mathbb{F}_2^{2^{n-1}}$ and $\sigma\in W_{n-1}$, we have
$$\sigma(v) = (v_{\sigma^{-1}(1)},\dots, v_{\sigma^{-1}(2^{n-1})}).$$
Hence, throughout the article, we will think of the elements of $W_n$ as pairs $(v,s)$, $v\in \mathbb{F}_2^{2^{n-1}}, s\in W_{n-1}$. With this notation, if $\sigma_1=(v,s),\sigma_2=(w,t)\in W_n$ are two elements in $W_n$, their product is given by
\begin{equation}
\label{eq1}
\sigma_1\sigma_2 = (v+s(w),st).
\end{equation}
If $G:=\{g_1,\dots,g_k\}\leq W_n$, where $g_i= (v_i,s_i)$ for $i=1,\dots,k$, we will also often use the identification
\begin{equation}
\label{eq2}
\pi_n(G) = \{s_1,\dots, s_k\}\leq W_{n-1}.
\end{equation}
By the iterated wreath product definition of $W_n$, it immediately follows that each $W_n$ is a $2$-group for all $n\geq 1$, thus are all its subgroups. In fact, it is well-known that $W_n$ is isomorphic to the Sylow $2$-subgroup of $\text{Sym}(2^n)$ for all $n\geq 1$, where $\text{Sym}(2^n)$ stands for the full symmetric group of degree $2^n$.

Recall that for a group $X$, \emph{the Frattini subgroup of X}, denoted by $\Phi(X)$, is defined as the intersection of all maximal subgroups of $X$. It is well-known that for a $2$-group $X$, $\Phi(X)$ is generated by squares and commutators, i.e. we have
\begin{equation}
\label{eq3}
\Phi(X) = X^2[X,X].
\end{equation}
Finally, we introduce the following notation, which will be very convenient in calculations throughout the paper.
\begin{defin}
\label{def2.1}
	Let $\sigma\in W_n$ for some $n\geq 1$. We define the set $\text{Fix}(\sigma)$ by
	$$\text{Fix}(\sigma)=\{v\in \mathbb{F}_2^{2^n}\text{ }|\text{ }\sigma(v)=v\}.$$
\end{defin}
\section{Auxiliary lemmas}
\label{sec:sec4}
In this section, we will prove several auxiliary lemmas that will be used in the proof of Theorem \ref{thm1.4}.
\begin{lemma}
\label{lem3.1}
For $n\geq 1$, let $H,G\leq W_n$ be two subgroups of $W_n$. Suppose that H is elementwise $K_n$-conjugate into $G$, $|H|=|G|$ and $H\cap K_n = \{\text{id}\}$. Then $G\cap K_n = \{\text{id}\}$.
\end{lemma}
\begin{proof}
Let $|H|=|G|=k$. If $H$ is $\ell$-generated, set $H=\langle x_1,x_2,\dots,x_{\ell}\rangle$. Then we can write $H=\{h_1,h_2,\dots,h_k\}$, where $h_i = w_i(x_1,x_2,\dots,x_{\ell})$ for some word $w_i$ in $x_1,x_2,\dots,x_{\ell}$ for $i=1,\dots,k$. Since $H$ is elementwise $K_n$-conjugate into $G$, $x_1^{a_1},x_2^{a_2},\dots,x_{\ell}^{a_{\ell}}\in G$ for some $a_1,a_2,\dots a_{\ell}\in K_n$. Note that $\pi_n(w_i(x_1,x_2,\dots,x_{\ell})) = \pi_n(w_i(x_1^{a_1},x_2^{a_2},\dots,x_{\ell}^{a_{\ell}}))$ since $a_1,\dots,a_{\ell}\in K_n=\text{Ker}(\pi_n)$, hence $\pi_n(w_i(x_1^{a_1},x_2^{a_2},\dots,x_{\ell}^{a_{\ell}})) \neq \pi_n(w_j(x_1^{a_1},x_2^{a_2},\dots,x_{\ell}^{a_{\ell}}))$ for $i\neq j$ since $H\cap K_n = \{\text{id}\}$. This gives that $|G|\geq |\pi_n(G)|\geq k = |G|$, hence $|\pi_n(G)| = |G|$, which proves $G\cap K_n = \{\text{id}\}$, as desired.
\end{proof}
\begin{lemma}
\label{lem3.2}
For $n\geq 1$, let $H,G\leq W_n$ be two subgroups of $W_n$. Suppose that H is elementwise $K_n$-conjugate into $G$, $|H|=|G|$ and $H\cap K_n = \{\text{id}\}$. Then for any $H_1\leq H$, there exists $G_1\leq G$ such that $|H_1|=|G_1|$ and $H_1$ is elementwise $K_n$-conjugate into $G_1$.
\end{lemma}
\begin{proof}
Let $H_1=\{h_1,\dots,h_k\}\leq H$. Since $H$ is elementwise $K_n$-conjugate into $G$, there exists $a_1,\dots,a_k\in K_n$ such that $h_i^{a_i}\in G$ for $i=1,\dots,k$. We set $G_1=\langle h_1^{a_1},\dots,h_k^{a_k}\rangle$. Clearly $H_1$ is $K_n$-conjugate into $G_1$. We also claim that $|H_1|=|G_1|$. To see this: Note that by Lemma \ref{lem3.1}, $G\cap K_n = \{\text{id}\}$, hence $G_1\cap K_n = \{\text{id}\}$. Since $a_1,\dots,a_k\in K_n$, we have $\pi_n(H_1) = \pi_n(G_1)$. Since we have $H_1\cap K_n = G_1\cap K_n = \{\text{id}\}$, this immediately implies $|H_1|=|G_1|$, as desired. 
\end{proof}
\begin{lemma}
\label{lemma_late1}
For $n\geq 1$, let $H,G\leq W_n$ be two non-cyclic subgroups of $W_n$ such that $H\cap K_n = \{\text{id}\}$, $|H|=|G|$, and $H$ is elementwise $K_n$-conjugate into $G$. For two distinct maximal subgroups $H_1, H_2\trianglelefteq H$, suppose that $H_1^a,H_2^b\leq G$ for some $a,b\in K_n$. Then $G=\langle H_1^a,H_2^b\rangle$.
\end{lemma}
\begin{proof}
Since $H_1$ and $H_2$ are distinct maximal subgroups of $H$, we clearly have $H=\langle H_1,H_2\rangle$. Note that since $a,b\in K_n$, we have $\pi_n(H)=\pi_n(\langle H_1,H_2\rangle)=\pi_n(\langle H_1^a,H_2^b\rangle)\leq \pi_n(G)$. Since $H\cap K_n=\{\text{id}\}$ and $G\cap K_n=\{\text{id}\}$ (by Lemma \ref{lem3.1}), this gives
$$|H|=|\pi_n(H)|=|\pi_n(\langle H_1,H_2\rangle)|=|\pi_n(\langle H_1^a,H_2^b\rangle)|\leq |\pi_n(G)|=|G|,$$
which, since $|H|=|G|$, implies that all quantities must be equal, and in particular $$|\pi_n(\langle H_1^a,H_2^b\rangle)|=|\pi_n(G)|.$$
Recalling that $G\cap K_n = \{\text{id}\}$ (thus $\langle H_1^a,H_2^b\rangle\cap K_n = \{\text{id}\}$), this gives $\langle H_1^a,H_2^b\rangle=G$, as desired.
\end{proof}
In the following lemma and throughout the rest of the article, for any subgroup $H\leq W_n$ of $W_n$, $C_{K_n}(H)$ denotes the group of elements of $K_n$ which commute with each element of $H$.
\begin{lemma}
\label{lem3.3}
For $n\geq 1$, let $H,G\leq W_n$ be two non-cyclic subgroups of $W_n$ such that $H\cap K_n = \{\text{id}\}$, $|H|=|G|$, and $H$ is elementwise $K_n$-conjugate into $G$. For two distinct maximal subgroups $H_1, H_2\trianglelefteq H$, suppose that $G=\langle H_1, H_2^a\rangle$ for some $a\in K_n$. Then $H$ is globally $K_n$-conjugate into $G$ if and only if $a\in C_{K_n}(H_1)C_{K_n}(x)$ for some $x\in H_2\setminus  H_1$.
\end{lemma}
\begin{proof}
We first assume that $a\in C_{K_n}(H_1)C_{K_n}(x)$ for some $x\in H_2\setminus  H_1$. Write $a=bc$ for $b\in C_{K_n}(H_1)$, $c\in C_{K_n}(x)$. Since $H_1$ has index 2 in $H$, we have $H=\langle H_1,x\rangle$. Since $\pi_n(x^a)=\pi_n(x)$, this gives $\pi_n(H) = \pi_n(\langle H_1, x^a\rangle)$. Since $H\cap K_n = G\cap K_n = \{\text{id}\}$ (by Lemma \ref{lem3.1}), this immediately implies that $G=\langle H_1, x^a\rangle$. We have $$H^b = \langle H_1^b, x^b\rangle = \langle H_1, x^{ac}\rangle = \langle H_1, x^a\rangle = G,$$
where we used the fact that $K_n$ is an elementary abelian $2$-group in second and third equalities. This finishes the proof of one direction.\par

We now assume that $H$ is globally $K_n$-conjugate into $G$, i.e. there exists $b\in K_n$ such that $H^b = G$ (since $|H|=|G|$). Consider an arbitrary element $x\in H_2\setminus H_1$. Note that $H=\langle H_1, H_2\rangle$. Since $G\cap K_n = \{\text{id}\}$ by Lemma \ref{lem3.1}, we immediately get $H_1^b=H_1$ and $H_2^b=H_2^a$, which, since $H_1\cap K_n = H_2\cap K_n = \{\text{id}\}$ and $a,b\in K_n$, implies that $b\in C_{K_n}(H_1)$ and $ab\in C_{K_n}(H_2)$. This gives $a=b(ab)\in  C_{K_n}(H_1)C_{K_n}(H_2)\leq C_{K_n}(H_1)C_{K_n}(x)$, as desired.
\end{proof}
\begin{lemma}
\label{lem3.4}
	For $n\geq 1$, let $H,G\leq W_n$ be two non-cyclic subgroups of $W_n$ such that $H\cap K_n = \{\text{id}\}$, $|H|=|G|$, and $H$ is elementwise $K_n$-conjugate into $G$. For two distinct maximal subgroups $H_1, H_2\trianglelefteq H$, suppose that $G=\langle H_1, H_2^a\rangle$ for some $a\in K_n$. Then $a\in C_{K_n}(\Phi(H))$.
\end{lemma}
\begin{proof}
Take $\alpha\in \Phi(H)$. Since $\alpha\in H_1\cap H_2$, we must have $\alpha, \alpha^a\in G$. Note that $\pi_n(\alpha) = \pi_n(\alpha^a)$. Since $G\cap K_n = \{\text{id}\}$ by Lemma \ref{lem3.1}, this immediately gives $\alpha = \alpha^a$, as desired.
\end{proof}
\begin{lemma}
\label{lem3.5}
Let $n\geq 1$. For $a\in K_n$, $x\in W_n$, write $a=(u,1)$, $x=(v,s)$ for some $u,v\in \mathbb{F}_2^{2^{n-1}}$, $s\in W_{n-1}$. Then $a\in C_{K_n}(x)\iff u\in \text{Fix}(s)$.
\end{lemma}
\begin{proof}
If we rewrite the equality $x^a = x$ using (\ref{eq1}), we obtain
$$(u,1)(v,s)(u,1) = (v,s) \iff (u+v+s(u),s) = (v,s).$$
This immediately gives $s(u) = u$, as desired.
\end{proof}
\begin{lemma}
\label{lemma_late2}
For $n\geq 1$, let $H\leq W_n$ be a subgroup of $W_n$. Suppose that $H\cap K_n = \{\text{id}\}$. Then $\pi_n(\Phi(H))=\Phi(\pi_n(H))$.
\end{lemma}
\begin{proof}
For any finite group $G$, let $m(G)$ denote the size of the minimal set of generators of $G$. Note that $m(H)=m(\pi_n(H))$ because $H$ and $\pi_n(H)$ are isomorphic by the assumption $H\cap K_n=\{\text{id}\}$.

For any finite 2-group $G$, by Burnside's Basis Theorem, we have 
\[|G/\Phi(G)|=\frac{|G|}{|\Phi(G)|}=2^{m(G)}.\]
Therefore, from above, we have $\frac{|H|}{|\Phi(H)|} = \frac{|\pi_n(H)|}{|\Phi(\pi_n(H))|}$. Since $|H|=|\pi_n(H)|$ by the assumption $H\cap K_n = \{\text{id}\}$, this gives
\begin{equation}
\label{eq_late1}
|\Phi(\pi_n(H))| = |\Phi(H)| = |\pi_n(\Phi(H))|,
\end{equation}
where we used the fact that $\Phi(H)\cap K_n = \{\text{id}\}$ in the last equality. On the other hand, considering the surjection $\pi_n:H \twoheadrightarrow\pi_n(H)$, it is a well-known (and easy) exercise to show that
\begin{equation}
\label{eq_late2}
\pi_n(\Phi(H))\leq \Phi(\pi_n(H)).
\end{equation}
We can now combine (\ref{eq_late1}) and (\ref{eq_late2}) to conclude that $\pi_n(\Phi(H))=\Phi(\pi_n(H))$, as desired.
\end{proof}
\begin{lemma}
\label{lem3.6}
	Let $x,y\in W_n$, $a,b\in K_n$. Suppose that $x,y,a,b$ satisfy the equality $$xy^a=(xy)^b.$$
	Setting $x=(v,s)$, $y=(w,t)$, $a=(u,1)$ and $b=(z,1)$, we have
	$$s(u)+st(u) = z+st(z).$$
\end{lemma}
\begin{proof}
	By direct computation using (\ref{eq1}), we get
	\begin{align*}
	xy^a &= (v,s)(u+w+t(u),t)\\
	&= (v+s(u)+s(w)+st(u),st).
	\end{align*}
	Similarly, we have
	$$(xy)^b = (z+v+s(w)+st(z),st)$$
	If we set the expressions for $xy^a$ and $(xy)^b$ equal to each other, we immediately get
	$$s(u)+st(u) = z+st(z),$$
	as desired.
\end{proof}
\section{Proof of Theorem 1.4}
\label{sec:sec5}
We will start with the following proposition, which will be the key ingredient in the inductive step in the proof of Theorem \ref{thm1.4}.
\begin{prop}
\label{prop4.1}
Let $X\leq W_n$ be a subgroup of $W_n$. Take $\alpha\in W_n\setminus X$ and also set $Y=\langle X,\alpha\rangle\leq W_n$. Suppose that a vector $v=(v_1,\dots,v_{2^n})\in \mathbb{F}_2^{2^n}$ satisfies the following three properties:
\begin{enumerate}[label=(\alph*)]
	\item $v_i + v_{\alpha(i)} = v_{\beta(i)} + v_{\beta\alpha(i)}$ for any $\beta\in X$ and $i\in \{1,\dots,2^n\}$.
	\item $v_i = v_j$ if $\text{Orb}_\alpha(i) = \text{Orb}_\alpha(j)$ and $\text{Orb}_{\beta}(i) = \text{Orb}_{\beta}(j)$ for some $\beta\in X$ and $i,j\in \{1,\dots,2^n\}$.
	\item $v\in \text{Fix}(\Phi(Y))$.
\end{enumerate}
Then $v\in \text{Fix}(\alpha)+\text{Fix}(X)$.
\end{prop}
\begin{proof}
We start by noting that, since $v\in \text{Fix}(\Phi(Y))$, throughout the proof, we will free to use $v_{\sigma(i)}$ instead of $v_{\sigma^{-1}(i)}$ for any $\sigma\in Y$ because $\sigma(\sigma^{-1})^{-1}=\sigma^2\in \Phi(Y)$ by (\ref{eq3}).

If $Y$ is cyclic, then we must have $Y=\langle \alpha\rangle$ since $\alpha\not\in X$. In this case, since $X\leq \langle \alpha\rangle$, condition (b) becomes \enquote{$v_i = v_j$ if $\text{Orb}_{\alpha}(i) = \text{Orb}_{\alpha}(j)$} and we also get $\text{Fix}(\alpha)+\text{Fix}(X) = \text{Fix}(X)$. Since any $v$ satisfying (b) clearly lies in $\text{Fix}(\alpha)\subseteq\text{Fix}(X)$, we are done.\par

We can now assume that $Y$ is non-cyclic. Take a vector $v\in \mathbb{F}_2^{2^n}$ that satisfies all three conditions. Consider an arbitrary $i\in \{1,\dots,2^n\}$, and its orbit $O(i):=O_{Y}(i)$ under the action of $Y$. By definition, for any $j,j'\in O(i)$, there exists $\sigma,\sigma'\in Y$ such that $j=\sigma(i)$, $j'=\sigma'(i)$. For any $y\in Y$, let $\bar{y}$ denote the image of $y$ under the quotient map $Y\twoheadrightarrow Y/\Phi(Y)$. Since $v$ satisfies condition (c), $v_j = v_{j'}$ if $\bar{\sigma} = \bar{\sigma'}$. Suppose that the group $Y$ is $\ell+1$-generated for some $\ell\geq 1$. Since $\alpha\not\in X$, we can write $Y=\langle \beta_1,\dots,\beta_{\ell},\alpha\rangle$ for some $\beta_1,\dots,\beta_{\ell}\in X$ with $X=\langle \beta_1,\dots,\beta_{\ell}\rangle$. By Burnside's Basis Theorem, $Y$ is a union of $2^{\ell+1}$ distinct cosets of $\Phi(Y)$. More concretely, we can write $Y$ as a disjoint union
$$Y = (\bigcup\limits_{i=1}^{2^{\ell}} A_i)\cup (\bigcup\limits_{i=1}^{2^{\ell}} \alpha A_i),$$
where each $A_i$ (hence $\alpha A_i$) is a coset of $\Phi(Y)$, and $X=\bigcup\limits_{i=1}^{2^{\ell}} A_i$.
Fix the representatives $i_1,\dots,i_k$ of distinct orbits of $\{1,\dots,2^n\}$ under the action of $Y$, i.e.
$$\bigcup\limits_{j=1}^{k} O(i_j) = \{1,\dots,2^n\}.$$
\begin{defin}
\label{def4.2}
Let $v\in \mathbb{F}_2^{2^n}$. We say \emph{$\mathcal{P}_Y(v)$ holds} if $v$ satisfies the following condition: 
\begin{itemize}
\item For any $\sigma,\sigma'\in Y$ and $j\in\{1,\dots,k\}$, $v_{\sigma(i_j)}=v_{\sigma'(i_j)}=a_{i_j}^{(t)}$ for some $a_{i_j}^{(t)}\in \{0,1\}$ if $\sigma\in A_t$, $\sigma'\in \alpha A_t$ for some $1\leq t\leq 2^{\ell}$.
\end{itemize}

\end{defin}
We define the subspace $V_Y\subseteq \mathbb{F}_2^{2^n}$ by
\begin{equation}
\label{eq4}
V_Y = \{u\in \mathbb{F}_2^{2^n}|\text{ }\mathcal{P}_Y(u)\text{ holds}.\}.
\end{equation}
The proof of Proposition \ref{prop4.1} will now follow from the next claim.
\begin{claim}
\label{claim4.3}
There exists $v'\in V_Y$ such that $v'\in \text{Fix}(\alpha)$ and $v-v'\in \text{Fix}(X)$.
\end{claim}
\begin{proof}[Proof of Claim \ref{claim4.3}]
Any vector in $V_Y$ already lies in $\text{Fix}(\alpha)$ by definition. Thus, to finish the proof, we need to find $v'\in V_Y$ such that $v-v'\in \text{Fix}(X)$.

For any $\sigma_1,\sigma_2\in Y$, we have
\[\sigma_1(\sigma_2)^{-1}\in \Phi(Y) \iff \sigma_1,\sigma_2\in A_t \text{ or }\sigma_1,\sigma_2\in \alpha A_t\text{ for some }t\in\{1,\dots,2^{\ell}\}.\]
Therefore, by definition, any $v'\in V_Y$ lies in $\text{Fix}(\Phi(Y))$. Since $v\in \text{Fix}(\Phi(Y))$ by assumption, we conclude that $v-v'\in \text{Fix}(\Phi(Y)).$

Also note that, since $X=\bigcup\limits_{i=1}^{2^{\ell}} A_i$, for any $\sigma_1,\sigma_2\in Y$, we have $$\sigma_1{\sigma_2}^{-1}\in X \iff \sigma_1\in A_t, \sigma_2\in A_{t'} \text{ or } \sigma_1\in \alpha A_t, \sigma_2\in \alpha A_{t'} \text{ for some }t,t'\in \{1,\dots,2^{\ell}\}.$$ Thus, it follows that for any $v'\in V_Y$, we have $v-v'\in \text{Fix}(X)$ if and only if the following two conditions are satisfied for $j=1,\dots,k$:
\begin{enumerate}
	\item  $(v-v')_{\sigma_1(i_j)}=(v-v')_{\sigma_2(i_j)}$ if $\sigma_1\in A_i, \sigma_2\in A_{i'}$ for some $i,i'\in \{1,\dots,2^{\ell}\}$.
	\item  $(v-v')_{\sigma_3(i_j)}=(v-v')_{\sigma_4(i_j)}$ if $\sigma_3\in \alpha A_i, \sigma_4\in \alpha A_{i'}$ for some $i,i'\in \{1,\dots,2^{\ell}\}$.
\end{enumerate}
We will now show that a vector $v'\in V_Y$ that satisfies both conditions indeed exists.

\textbf{Condition (1)}: With the notation in Definition \ref{def4.2}, condition (1) is equivalent to the equality
\begin{equation}
\label{eq5}
v_{\sigma_1(i_j)}+v_{\sigma_2(i_j)}=a_{i_j}^{(i)}+a_{i_j}^{(i')}
\end{equation}
for some $a_{i_j}^{(i)},a_{i_j}^{(i')}\in \{0,1\}$. Note that if the variables $a_{i_j}^{(i)},a_{i_j}^{(i')}\in \{0,1\}$ are independent, then we can choose them so that (\ref{eq5}) holds. If they are not independent, this gives us four different cases:
\begin{enumerate}
\item $\pi_1(i_j)=\pi_2(i_j)$ for some $\pi_1\in A_i,\pi_2\in A_{i'}$. In this case, $a_{i_j}^{(i)}+a_{i_j}^{(i')}=0$ by Definition \ref{def4.2}. Recalling that $v\in \text{Fix}(\Phi(Y))$, the left-hand side of (\ref{eq5}) also vanishes since $v_{\sigma_1(i_j)}=v_{\pi_1(i_j)}=v_{\pi_2(i_j)}=v_{\sigma_2(i_j)}$, where we used Definition~\ref{def4.2} in the first and the third equalities. Hence, (\ref{eq5}) holds.
\item $\pi_1(i_j)=\alpha \pi_2(i_j)$ for some $\pi_1\in A_i,\pi_2\in A_{i'}$. In this case, we have $\text{Orb}_{\alpha}(\pi_1(i_j))=\text{Orb}_{\alpha}(\pi_2(i_j))$ and $\text{Orb}_{\pi_1(\pi_2)^{-1}}(\pi_1(i_j))=\text{Orb}_{\pi_1(\pi_2)^{-1}}(\pi_2(i_j))$. Thus, since $\pi_1(\pi_2)^{-1}\in X$, condition (2) of Proposition \ref{prop4.1} implies that $v_{\pi_1(i_j)}=v_{\pi_2(i_j)}$. Recalling that $v\in \text{Fix}(\Phi(Y))$, and that $\pi_1\sigma_1^{-1},\pi_2\sigma_2^{-1}\in \Phi(Y)$, this immediately gives $v_{\sigma_1(i_j)}=v_{\sigma_2(i_j)}$ since $v_{\sigma_1(i_j)}=v_{\pi_1(i_j)}$ and $v_{\sigma_2(i_j)}=v_{\pi_2(i_j)}$. Therefore, to show that  (\ref{eq5}) holds, it remains to see that $a_{i_j}^{(i)}+a_{i_j}^{(i')}$ vanishes. But, this is immediate by Definition \ref{def4.2} and the condition $\pi_1(i_j)=\alpha \pi_2(i_j)$, hence we are done.
\item $\alpha\pi_1(i_j)=\alpha\pi_2(i_j)$ for some $\pi_1\in A_i,\pi_2\in A_{i'}$. Applying $\alpha^{-1}$ to both sides of $\alpha\pi_1(i_j)=\alpha\pi_2(i_j)$, we obtain case (1), hence the result follows from case (1).
\item $\alpha\pi_1(i_j)=\pi_2(i_j)$ for some $\pi_1\in A_i,\pi_2\in A_{i'}$. Applying $\alpha^{-1}$ to both sides of $\alpha\pi_1(i_j)=\pi_2(i_j)$, we obtain $\pi_1(i_j)=\alpha^{-1}\pi_2(i_j)$. By definition of $A_{i'}$, since $\alpha(\alpha^{-1})^{-1}=\alpha^2\in \Phi(Y)$ by (\ref{eq3}), there exists $\pi_3\in A_{i'}$ such that $\alpha^{-1}\pi_2(i_j)=\alpha\pi_3(i_j)$. This yields to $\pi_1(i_j) = \alpha\pi_3(i_j)$, and the result now follows from case (2).
\end{enumerate}

\textbf{Condition (2):} This condition is clearly equivalent to the condition
\begin{equation}
\label{eq6}
(v-v')_{\alpha\sigma_1(i_j)} = (v-v')_{\alpha\sigma_2(i_j)} \text{ if }\sigma_1\in A_i, \sigma_2\in A_{i'}\text{ for some }i,i'\in1,\{\dots,2^{\ell}\}.
\end{equation}
Applying $\alpha$ to both sides of the equality in (\ref{eq6}), we obtain
$$(v-v')_{\sigma_1(i_j)} = (v-v')_{\sigma_2(i_j)},$$
which is same as condition (1). Hence, the proof follows from the first part. We have covered all possible cases, hence we are done.
\end{proof}
The proof of Proposition \ref{prop4.1} now follows from Claim \ref{claim4.3}.
\end{proof}
The following will be crucial in showing that certain vectors that will appear in the proof of Theorem \ref{thm1.4} do indeed satisfy the conditions given in Proposition \ref{prop4.1}, hence the conclusion of Proposition \ref{prop4.1}. 
\begin{prop}
\label{prop4.4}
Let $X\leq W_n$ be a subgroup of $W_n$. Take $\alpha\in W_n\setminus X$ and also set $Y=\langle X,\alpha\rangle\leq W_n$. Suppose that a vector $v\in \mathbb{F}_2^{2^n}$ satisfies the following conditions:
\begin{enumerate}
\item For any $\beta\in X$, there exists $u^{(\beta)}\in \text{Fix}(\Phi(Y))$ such that $\alpha(v)+\alpha\beta(v)=u^{(\beta)}+\alpha\beta(u^{(\beta)})$.
\item $v\in\text{Fix}(\Phi(Y))$.
 
\end{enumerate}
Then we have $v\in\text{Fix}(\alpha)+\text{Fix}(X)$.
\end{prop}
\begin{proof}
To prove the proposition, it suffices to show that the vector $v$ satisfies the conditions (a) and (b) in Proposition \ref{prop4.1}. Note that we have the equality
\begin{equation}
\label{eq7}
\alpha(v)+\alpha\beta(v)=u^{(\beta)}+\alpha\beta(u^{(\beta)}).
\end{equation}
If we apply $\alpha\beta$ to both sides of (\ref{eq7}), since $v,u^{(\beta)}\in \text{Fix}(\Phi(Y))$ and $(\alpha\beta)^2,(\alpha\beta\alpha)(\beta)^{-1}\in \Phi(Y)$, we obtain
\begin{equation}
\label{eq8}
v+\beta(v) = u^{(\beta)}+\alpha\beta(u^{(\beta)}).
\end{equation}
Adding (\ref{eq7}) and (\ref{eq8}) yields to
\begin{equation}
\label{eq9}
v+\beta(v) = \alpha(v)+\alpha\beta(v)
\end{equation}
Comparing the coordinates of both sides in (\ref{eq9}), we obtain
\[v_i + v_{\beta^{-1}(i)} = v_{\alpha^{-1}(i)} + v_{(\alpha\beta)^{-1}(i)}\]
for any $i\in \{1,\dots,2^n\}$. Recalling that $v\in \text{Fix}(\Phi(Y))$, this can be rewritten as
\[v_i + v_{\beta(i)} = v_{\alpha(i)} + v_{\alpha\beta(i)}\]
for any $i\in \{1,\dots,2^n\}$.
Since this can be done for any $\beta\in X$, it follows that $v$ satisfies the condition (a) of Proposition \ref{prop4.1}.

It now remains to show that the vector $v$ satisfies condition (b) of Proposition~\ref{prop4.1}. To that end, suppose that for some $\beta\in X$, $\text{Orb}_{\alpha}(i)=\text{Orb}_{\alpha}(j)$ and $\text{Orb}_{\beta}(i)=\text{Orb}_{\beta}(j)$ for some $i,j\in \{1,\dots, 2^n\}$. Write $\alpha^a(i)=\beta^b(i)=j$ for some $a,b\in \mathbb{N}$. If at least one of $a$ or $b$ is even, then at least one of $\alpha^a$ or $\beta^b$ lies in $\Phi(Y)$ by (\ref{eq3}), which immediately gives that $v_i=v_j$ since $v\in \text{Fix}(\Phi(Y))$. Thus, we can assume without loss of generality that $a$ and $b$ are both odd. By comparing coordinates of both sides in (\ref{eq8}) and using the fact that $v,u^{(\beta)}\in \text{Fix}(\Phi(Y))$, we have
\begin{equation}
\label{eq10}
v_i+v_{\beta(i)}=u_i^{(\beta)}+u_{\alpha\beta(i)}^{(\beta)}.
\end{equation}
Since $v\in \text{Fix}(\Phi(Y))$, and $\beta^{b-1}\in \Phi(Y)$ by (\ref{eq3}) (recall that $b$ was odd), we have $v_{\beta(i)}=v_{\beta^b(i)}=v_j$, which, using (\ref{eq10}), yields to
\begin{equation}
\label{eq11}
v_i+v_j=u_i^{(\beta)}+u_{\alpha\beta(i)}^{(\beta)}.
\end{equation}
We also have $\alpha\beta(i)=\alpha\beta^{1-b}\beta^b(i)=\alpha\beta^{1-b}(j)$. Recalling that $u^{(\beta)}\in \text{Fix}(\Phi(Y))$ and that $1-a$ and $1-b$ are even, this gives
$$u^{(\beta)}_{\alpha\beta(i)}=u^{(\beta)}_{\alpha\beta^{1-b}(j)}=u^{(\beta)}_{\alpha(j)}=u^{(\beta)}_{\alpha^{-a}(j)} = u^{(\beta)}_i.$$
Finally, putting this in (\ref{eq11}), we obtain $v_i+v_j = u^{(\beta)}_i+u^{(\beta)}_i = 0$, hence $v_i=v_j$, as desired.
\end{proof}
We are finally ready to prove Theorem \ref{thm1.4}.
\begin{proof}[Proof of Theorem \ref{thm1.4}]
We will do induction on $|H|$. Note that the statement trivially holds for $|H|=1$ and $|H|=2$ since in those cases $H$ is cyclic. We now assume that the statement holds for $|H|=2^k$ for some $k\geq 1$, and to finish the inductive step, we will prove the statement for $|H|=2^{k+1}$. By the induction assumption and Lemma \ref{lem3.2}, if we take any maximal ideal $H_1\trianglelefteq H$, there exists $G_1\trianglelefteq G$ and $a\in K_n$ such that $H_1^a=G_1.$ Thus, conjugating $G$ by $a$ if necessary, we can assume without loss of generality that $H_1\trianglelefteq G$. Note that this also implies $\Phi(H)\leq G$ since $\Phi(H)$ is contained in $H_1$. The statement already holds if $H$ is cyclic, so we can assume without loss of generality that $H$ is non-cyclic, which gives two distinct maximal subgroups $H_1,H_2$ of $H$ with $H=\langle H_1, H_2\rangle$. By the induction assumption and Lemma \ref{lem3.2}, there exists $b\in K_n$ such that $H_2^b\trianglelefteq G$. By Lemma \ref{lemma_late1}, this yields to $G=\langle H_1,H_2^b\rangle$. By Lemma \ref{lem3.4}, then, we obtain $b\in C_{K_n}(\Phi(H))$. To prove that $H$ and $G$ are conjugate under $K_n$, by Lemma \ref{lem3.3}, we need to show that $b\in C_{K_n}(H_1)C_{K_n}(x)$ for some $x\in H_2\setminus H_1$. Setting $b=(u,1)$, $x=(v,s)$ for $u,v\in \mathbb{F}_2^{2^{n-1}}$, $s\in W_{n-1}$, by Lemma \ref{lem3.5}, this is equivalent to showing that
\begin{equation}
\label{eq12}
u\in \text{Fix}(\pi_n(H_1))+\text{Fix}(s),
\end{equation}
where we used the identification given in (\ref{eq2}) when writing $\text{Fix}(\pi_n(H))$.

Fix some $x\in H_2\setminus H_1$. For each $h_i\in H_1$, if we consider a maximal subgroup $H^{(i)}\trianglelefteq H$ of $H$ that contains $h_ix$, by the induction assumption and Lemma \ref{lem3.2}, there exists $c_i\in K_n$ such that $(H^{(i)})^{c_i}$ is a maximal subgroup of $G$. Since $G\cap K_n = \{\text{id}\}$ (by Lemma \ref{lem3.1}), $\Phi(H)\leq H_1\leq G$ and $\Phi(H)^{c_i}\leq (H^{(i)})^{c_i}\leq G$, this immediately gives $\Phi(H) = \Phi(H)^{c_i}$, i.e. $c_i\in C_{K_n}(\Phi(H))$. Recall that $H_2^b\leq G$, thus $x^b\in G$. Since both $(h_ix)^{c_i}$ and $h_ix^b$ lie in $G$, $\pi_n(h_ix^b)=\pi_n((h_ix)^{c_i})$, and $G\cap K_n = \{\text{id}\}$, we obtain
\begin{equation}
\label{eq13}
h_ix^b = (h_ix)^{c_i}
\end{equation}
for $i=1,\dots, |H_1|=2^k$. Setting $h_i = (w_i,t_i)$, $c_i = (z_i,1)$ (and recalling $x=(v,s)$, $b=(u,1)$), using Lemma \ref{lem3.6}, (\ref{eq13}) yields to
\begin{equation}
\label{eq14}
t_i(u)+t_is(u) = z_i + t_is(z_i)
\end{equation}
for $i=1,\dots 2^k$. Recall from above that $c_i\in C_{K_n}(\Phi(H))$, i.e., by Lemma \ref{lem3.5}, $z_i\in \text{Fix}(\pi_n(\Phi(H))) = \text{Fix}(\Phi(\pi_n(H)))$, where we used Lemma \ref{lemma_late2} in the last equality. Also recall from the first part of the proof that $b\in C_{K_n}(\Phi(H))$, i.e., by Lemma \ref{lem3.5}, $u\in \text{Fix}(\Phi(\pi_n(H)))$. If we now take $X=\langle t_1,\dots,t_{2^k}\rangle$, $Y=\pi_n(H)$, $\alpha=s$, $v=u$ in Proposition \ref{prop4.4}, recalling (\ref{eq14}), the vector $u$ satisfies both conditions in Proposition~\ref{prop4.4}. It immediately follows that
\[u\in \text{Fix}(\la t_1,\dots,t_{2^k}\ra)+\text{Fix}(s)=\text{Fix}(\pi_n(H_1))+\text{Fix}(s),\] 
which completes the proof of Theorem~\ref{thm1.4} by (\ref{eq12}).
\end{proof}
\section{A conjecture}
\label{sec:sec6}
We close the paper by a conjecture based on extensive Magma computations.
\begin{conj}
\label{conj:odom}
For $n\geq 1$, let $H\leq W_n$ be a subgroup of $W_n$. Suppose that $H$ contains a permutation $\sigma\in W_n$ that acts transitively on $\{1,\dots,2^n\}$. Then $\mathcal{P}(H,G)$ holds for any subgroup $G\leq W_n$.
\end{conj}
The author's interest in Conjecture~\ref{conj:odom} comes from the following: It has the potential of being particularly useful for the number theory application mentioned in Section~\ref{sec:sec2}, because for most polynomials $f$, for $n\geq 1$, the group $G_n(f)$ \emph{does} contain an element that acts transitively on $\{1,\dots,2^n\}$. Thus, in such cases, Conjecture~\ref{conj:odom} will imply that a conjugate of $G_n(f)$ is contained in $M_n(f)$ as long as $G_n(f)$ is elementwise $K_n$-conjugate into $M_n(f)$.
\subsection*{Acknowledgments}
The author would like to thank Nigel Boston for many helpful discussions related to the material in this paper. The author also thanks Rafe Jones for a helpful discussion related to this work. Finally, the author thanks the anonymous referee for the comments which helped improve the exposition.
	
\end{document}